 \documentclass[11pt]{amsart}

\usepackage[colorlinks, pagebackref, linkcolor=red, citecolor=blue]{hyperref}
\numberwithin{equation}{section}
\theoremstyle{plain}
\newtheorem{theorem}{Theorem}[section]
\newtheorem{lemma}[theorem]{Lemma}

 {\theoremstyle{definition}}

\def\bs {\mathbf}

\def\R{ \mathbb R}

\def\D{{ \mathbb D}}
\def\C{{ \mathbb C}}

\def\N{{ \mathbb N}}

\def\dbar{\ov\partial}

\def\Inter{\bigcap }
\def\ov{\overline}

\def\ss{\subseteq}
\def\emp{\emptyset}

\def\buildrel#1_#2^#3{\mathrel{\mathop{\kern 0pt#1}\limits_{#2}^{#3}}}

\begin{document}

%
%
%
%
%
%
%
%
%

\title[Cartan's Nullstellensatz]
 {A short proof of Cartan's Nullstellensatz for entire functions in $\C^n$}

\author{Raymond Mortini}

\address{%
Universit\'{e} de Lorraine\\
 D\'{e}partement de Math\'{e}matiques et  
Institut \'Elie Cartan de Lorraine,  UMR 7502\\
 Ile du Saulcy\\
 F-57045 Metz, France} 
 \email{raymond.mortini@univ-lorraine.fr}

\subjclass{Primary 32A15; Secondary 46J15}

\keywords{Entire functions; Cartan's Nullstellensatz; polydisk algebra; maximal ideals}

\date{June 1, 2015}

\begin{abstract}
Using the fact that the maximal ideals in the polydisk algebra are given by the kernels of
point evaluations, we derive a simple formula that gives a solution to the B\'ezout equation
in the space of all entire functions of several complex variables. Thus 
a short and easy analytic proof of Cartan's Nullstellensatz is obtained.

\end{abstract}

\maketitle

\section{Introduction}

The aim of this note is to give a short and easy proof of Cartan's Nullstellensatz:
\begin{theorem}\label{cartan}
Let $H(\C^n)$ be  the space of  functions holomorphic in $\C^n$. Given $f_j\in H(\C^n)$,
the B\'ezout equation $\sum_{j=1}^N g_jf_j=1$ admits a solution $(g_1,\dots,g_N)\in H(\C^n)^N$ 
if and only if the functions $f_j$
have no common zero in $\C^n$.
\end{theorem}
The usual proofs use a lot of machinery from sheaf theory, cohomology,  (see for example
\cite{kra}), or are based on the H\"ormander-Wolff method by solving higher order 
$\dbar$-equations using the Koszul complex, a tool  from
homological algebra (see  \cite[p. 128-131]{saw}). The one-dimensional case, 
first done by Wedderburn,
is very easy (see for example  \cite[p. 118-120]{rem} for the classical 
approach or \cite[p. 130]{an} for  the $\dbar$-approach).  For our proof to work,  we  shall only use 
 a standard fact from an introductory course to functional analysis, namely Gelfand's main theorem:
the  maximal ideals in a commutative unital  complex Banach algebra coincide  with the kernels of the
multiplicative linear functionals (see for instance \cite{rud}). 
The idea is to apply  this result to the polydisk algebras on an increasing sequence 
of polydisks $\bs D_k$ and to glue together the solutions
to the B\'ezout equations $\sum_{j=1}^N g_jf_j=1$ on $\bs D_k$ by using a Mittag-Leffler type trick.
The major hurdle to overcome was of course to find suitable summands that guarantee at the end
the holomorphy.

\section{The general solution to the B\'ezout equation}
Let $\D$ be the unit disk and 
let $A(\bs D^n)$ be the polydisk algebra; that is the algebra of those  continuous functions on the 
closed polydisk 
$$\bs D^n=\{(z_1,\dots,z_n)\in \C^n: |z_j|\leq 1\}$$
 which  are holomorphic in $\D^n$. Endowed with
the supremum norm, $A(\bs D^n)$ becomes a uniform algebra and coincides with the closure on
 $\bs D^n$ of the polynomials in $\C[z_1,\dots,z_n]$.  We actually only need 
 that $A(\bs D^n)$ is  the uniform algebra generated by the coordinate functions 
 $Z_j$, $j=1,\dots,n$ on $\bs D^n$. It is now straightforward to show that
 an ideal $I$ in $A(\bs D^n)$ is maximal if and only if it coincides with
 $M(a_0)=\{f\in A(\bs D^n): f(a_0)=0\}$ for some $a_0\in \bs D^n$
 (just take a character $m$ on $A(\bs D^n)$ and put $a_0=(m(Z_1),\dots, m(Z_n))$).
 Hence the B\'ezout equation $\sum_{j=1}^N x_jf_j$ in $A(\bs D^n)$ has a solution if and only
 if $\Inter_{j=1}^N Z_{\bs D^n}(f_j)=\emp$, where $Z_{\bs D^n}(f)$ is the zero set of $f$ 
 on $\bs D^n$.
 We will use the following well-known elementary result. For the reader's convenience we reproduce
 the proof here (see \cite{mowi}),
  because its understanding  is fundamental  for our  construction.
 
  \begin{lemma}\label{solbez}
Let $R$ be a commutative unital ring. 
Suppose that $\bs a=(a_1,\dots,a_N)$ is an invertible $N$-tuple in $R^N$
 and let $\bs x=(x_1,\dots,x_N)$ satisfy
  $\sum_{j=1}^N x_j a_j=\bs 1$; that is $\bs x\,\bs a^t=\bs 1$.
 Then every other representation $\bs 1=\sum_{j=1}^N y_ja_j $ of $\bs 1$ can be deduced from
 the former by letting $\bs y=\bs x+ \bs a H$, where $H$ is an antisymmetric 
 $(N\times N)-$matrix over $R$;
 that is $H=-H^t$, where $H^t$ is the transpose of $H$.
\end{lemma}
\begin{proof}
Suppose that  $\bs 1=\bs x \bs a^t$ and $\bs 1=\bs y \bs a^t$. 
For $k=1,\dots, N$, multiply the first equation   by $y_k$  and the second  by $x_k$.  
Then $$x_k-y_k= \sum_{j\not=k}a_j (y_jx_k-y_kx_j).$$
 Thus $\bs y=  \bs x+ \bs a H $ for some
antisymmetric matrix $H$.

To prove the converse, let   $\bs 1=\bs x \bs a^t$.  Since $H$ is antisymmetric we have
(due to the transitivity of matrix multiplication and $\bs x\bs y^t=\bs y \bs x^t$)
\begin{eqnarray*}
(\bs aH)  \bs a^t &=& \bs a (H\bs a^t)\;\;\,=\;\bs a ( \bs a H^t)^t\\
&=&\bs a (-\bs a H)^t=(-\bs aH) \bs a^t.
\end{eqnarray*}
Thus  $(\bs aH)  \bs a^t=0$.
Hence
$$\bs y\bs a^t= (\bs x+ \bs a H)  \bs a^t= \bs x  \bs a^t +
 (\bs a H)  \bs a^t=\bs 1+0=\bs 1.$$
\end{proof}

 \section{Proof of  Theorem \ref{cartan}}
 
 \begin{proof}
 Let $f_j\in H(\C^n)$ and  put  $\bs f=(f_1,\dots,f_N)$.
Suppose that $\Inter_{j=1}^N Z(f_j)=\emp$. For $k\in \N^*$, let
$\bs D_k=(k\,\bs D)^n$ be the closed polydisk
$$\bs D_k=\{(z_1,\dots,z_n)\in \C^n: |z_j|\leq k\}.$$
Note that $\bs D_k\ss \bs D_{k+1}^\circ$.
Let $\bs a_k\in A(\bs D_{k+1})^N$ be a solution to the B\'ezout equation $\bs a_k\cdot \bs f^t=1$ on $\bs D_{k+1}$.  Using Tietze's extension theorem \footnote{ Since we will consider a telescoping series
$\sum T_j$, where the domains of definition of the summands $T_j$ are strictly increasing, even an application of Tietze's theorem is not necessary.}, we may assume that the tuples $\bs a_k$
have been continuously extended to $\C^n$ \footnote{ Of course, outside $\bs D_k$ the B\'ezout equation
is not necessarily satisfied.}.
 By Lemma \ref{solbez}, there is an antisymmetric  matrix $H_k$ over $A(\bs D_{k+1})$ such that
$$ \bs a_{k+1}=\bs a_k + \bs f \cdot H_k.$$
Put $\bs a_0=\bs 0$ and $H_0=\bf O$.  For $k=0,1,\dots$, 
let $P_k$ be an antisymmetric $N\times N$-matrix of polynomials in $\C[z_1,\dots, z_n]$ such that
$$\max_{\bs D_{k+1}}||\bs f\cdot H_k-\bs f\cdot P_k||_N< 2^{-k}.$$
We claim that the $N$-tuple
$$\bs g:= \sum_{k=0}^\infty \bigl(\bs a_{k+1}-\bs a_{k}-\bs f\cdot P_k\bigr)$$
belongs to $H(\C^n)^N$ and is a solution to the B\'ezout equation $\bs g\cdot\bs f^t=1$ in $H(\C^n)$.
In fact, let $\bs D_m$  be fixed. Then the series defining $\bs g$ is uniformly convergent on $\bs D_m$
since
\begin{eqnarray*}
\bs g&=& \sum_{k=0}^m \bigl(\bs a_{k+1}-\bs a_{k}-\bs f\cdot P_k\bigr) +\sum_{k=m+1}^\infty
\bigl(\bs a_{k+1}-\bs a_{k}-\bs f\cdot P_k\bigr)\\
&=& \bs a_{m+1}- \bs f\cdot \bigl(\sum_{k=0}^m P_k\bigr)
+\sum_{k=m+1}^\infty\bigl(\bs a_{k+1}-\bs a_{k}-\bs f\cdot P_k\bigr)
\end{eqnarray*}
and the tail can be majorated on $\bs D_m$  by 
$$\sum_{k=m+1}^\infty \Vert\bs f\cdot H_k -\bs f \cdot P_k||_N<2^{-m}.$$
Moreover, on $\bs D_m$, $\bs a_{m+1}$ and all the summands in the series $\sum_{k=m+1}^\infty$ are holomorphic. Since $m$ was arbitrarily
 chosen, we conclude that  $\bs g\in H(\C^n)^N$. 
 Moving again to $\bs D_m$ we see that, due to the antisymmetry of the  matrices $H_k$ and $P_k$,
\begin{eqnarray*}
\bs g\cdot \bs f^t&=& \bs a_{m+1}\cdot \bs f^t - \bs f\cdot \bigl(\sum_{k=0}^m P_k\bigr)\cdot \bs f^t
+\sum_{k=m+1}^\infty \bs f \cdot (H_k-P_k)\cdot \bs f^t\\
&=&1-0+0=1.
\end{eqnarray*} 
\end{proof}

\bigskip

{\bf  Acknowledgements}

I thank the referee of the journal ``American Math. Monthly" for some useful comments
concerning the introductory section  of a previous version of the paper.


\begin{thebibliography}{99}

\bibitem{an} M. Andersson. \emph{Topics in Complex Analysis},
Springer, New York 1997.

\bibitem{kra} S. Krantz, 
{\em Function theory of several complex variables},
 Reprint of the 1992 edition. AMS Chelsea Publishing, Providence, RI, 2001.
 
 \bibitem{mowi} R. Mortini, B. Wick,
 Simultaneous stabilization in $A_\R(\D)$,
 Studia Math. 191 (2009), 223--235.
 
\bibitem{rem} R. Remmert. \emph{Funktionentheorie II}, Springer, Berlin , 1991.

 \bibitem{rud} W. Rudin, \emph{Real and Complex Analysis},  third edition,
  McGraw-Hill, New York, 1986
  
  \bibitem{saw} E. Sawyer,
\emph{Function Theory: Interpolation and Corona Problems},
Fields Institute Monographs,  Amer. Math. Soc.  2009.


\end{thebibliography}
\end{document}